\documentclass[11pt]{amsart}

\newif\ifdraft\draftfalse

\usepackage{amssymb, enumerate, xspace, graphicx, url}
\usepackage[mathscr]{eucal}
\usepackage[all]{xy}
\SelectTips{cm}{}

\ifdraft\usepackage[notcite, notref]{showkeys}\fi

1

\allowdisplaybreaks[1]

\newenvironment{enumeratea}
{\begin{enumerate}[\upshape (a)]}
{\end{enumerate}}

\newtheorem*{namedtheorem}{\theoremname}
\newcommand{\theoremname}{testing}

\newcommand\Fields{\operatorname{Fields}}

\newcommand\Sets{\operatorname{Sets}}

\newcommand\chr{\operatorname{char}}

\newcommand\Spec{\operatorname{Spec}}
\newcommand\Char{\operatorname{char}}
\newcommand\Pf{\operatorname{Pf}}

\newtheorem{theorem}[equation]{Theorem}
\newtheorem{proposition}[equation]{Proposition}
\newtheorem{proposition-definition}[equation]{Proposition-Definition}

\newtheorem{lemma}[equation]{Lemma}

\theoremstyle{definition}

\newtheorem{remark}[equation]{Remark}

\theoremstyle{remark}

 \newcommand\ZZ{\mathbb{Z}}

\newcommand\rC{\mathrm{C}}

 \newcommand\rT{\mathrm{T}}
 
\newcommand\rW{\mathrm{W}}

\newcommand\arr{\ifinner \to\else\longrightarrow\fi}

\newcommand\arrto{\ifinner\mapsto\else\longmapsto\fi}

\renewcommand\H{\operatorname{H}}

\newcommand\eqdef{\overset{\mathrm{\scriptscriptstyle def}} =}

\def\displaytimes_#1{\mathrel{\mathop{\times}\limits_{#1}}}

\def\displayotimes_#1{\mathrel{\mathop{\bigotimes}\limits_{#1}}}

\newcommand\rank{\operatorname{rank}}

\newcommand\generate[1]{\langle #1 \rangle}
\newcommand\pfister[1]{{\ll #1 \gg}}

\newdir{ >}{{}*!/-5pt/@{>}}

\newcommand\doublelong[2]{\mathbin{\xymatrix{{}\ar@<3pt>[r]^{#1}
\ar@<-3pt>[r]_{#2}&}}}

\newlength{\ignora}

\newcommand{\ed}{\operatorname{ed}}

\newcounter{steps}

\newcommand{\trdeg}{\operatorname{tr\,deg}}

\newcommand{\mmu}{\boldsymbol{\mu}}

\newcommand\Spin{\mathbf{Spin}}

\newcommand\HSpin{\mathbf{HSpin}}
\newcommand\spin{\mathbf{Spin}}
\newcommand\SO{\mathbf{SO}}
\newcommand\Orth{\mathbf{O}}

\newcommand\Stab{\operatorname{Stab}}

\newcommand{\da}[1]{\pfister{\mspace{-3mu}{#1}\mspace{-3mu}}}

\begin{document}

\title[Essential dimension, spinor groups and 
quadratic forms]{Essential dimension,
spinor groups,\\and quadratic forms}

\author[Brosnan]{Patrick Brosnan$^\dagger$}

\author[Reichstein]{Zinovy Reichstein$^\dagger$}

\author[Vistoli]{Angelo Vistoli$^\ddagger$}

\address[Brosnan, Reichstein]{Department of Mathematics\\
The University of British Columbia\\
1984 Mathematics Road\\
Vancouver, B.C., Canada V6T 1Z2}

\address[Vistoli]{Scuola Normale Superiore\\Piazza dei Cavalieri 7\\
56126 Pisa\\Italy}

\email[Brosnan]{brosnan@math.ubc.ca}
\email[Reichstein]{reichst@math.ubc.ca}
\email[Vistoli]{angelo.vistoli@sns.it}

\begin{abstract} 
  We prove that the essential dimension of 
  the spinor group $\Spin_n$ grows exponentially with $n$ 
  and use this result to show that quadratic forms with 
  trivial discriminant and Hasse-Witt invariant 
  are more complex, in high dimensions, than previously expected.
\end{abstract}

\thanks{$^\dagger$Supported in part by an NSERC discovery grants and by
PIMS Collaborative Research Group in Algebraic geometry, cohomology 
and representation theory}
\thanks{$^\ddagger$Supported in part by the PRIN Project ``Geometria
sulle variet\`a algebriche'', financed by MIUR}

\maketitle


\section{Introduction}
\label{s.intro}

Let $K$ be a field of characteristic different 
from $2$ containing a square root of $-1$,
$\rW(K)$ be the Witt ring of $K$ and $I(K)$ be 
the ideal of classes of even-dimensional forms 
in $\rW(K)$; cf.~\cite{lam}. By abuse of notation,
we will write $q \in I^a(K)$ if the Witt class on
the non-degenerate quadratic form $q$ defined over $K$
lies in $I^{a}(K)$.
%
It is well known that every $q \in I^{a}(K)$ can be expressed as 
a sum of the Witt classes of $a$-fold Pfister forms 
defined over $K$; see, e.g.,~\cite[Proposition II.1.2]{lam}.
If $\dim(q) = n$, it is natural to ask how many Pfister forms 
are needed.  When $a = 1$ or $2$, it is 
easy to see that $n$ Pfister forms always 
suffice; see Proposition~\ref{p.PfisterEasy}. 
In this paper we will prove the following result, 
which shows that the situation is quite different 
when $a = 3$.

\begin{theorem} \label{t.pfister} 
Let $k$ be a field of characteristic different 
from $2$ and $n \ge 2$ be an even integer.  Then there 
is a field extension $K/k$ and an $n$-dimensional 
quadratic form $q \in I^3(K)$ 
with the following property:  for any finite 
field extension $L/K$ of odd degree 
$q_L$ is not Witt equivalent to the sum of fewer than
   \[ \frac{2^{(n+4)/4} -n-2}{7}  \]
$3$-fold Pfister forms over $L$.
\end{theorem}

Our proof of Theorem~\ref{t.pfister} is based on the new results on
the essential dimension of the spinor groups $\Spin_n$ in 
Section~\ref{sect.spin}, which are of independent interest. 

\subsection*{Acknowledgements}
We would like to thank the Banff International Research Station in
Banff, Alberta (BIRS) for providing the inspiring meeting place where
this work was started.  We are grateful to A.~Merkurjev and
B.~Totaro for bringing the problem of computing the Pfister 
numbers $\Pf_k(a, n)$ to our attention and for contributing 
Proposition~\ref{p.PfisterEasy2}. We also thank N.~Fakhruddin 
for helpful correspondence.

\section{Essential dimension}
\label{s.ed}

Let $k$ be a field. We will write $\Fields_k$ for the category of
field extensions $K/k$.  Let $F\colon\Fields_k \arr\Sets$ be 
a covariant functor.

  Let $L/k$ be a field extension. We will say that  $a\in F(L)$
  \emph{descends} to an intermediate field $k \subseteq 
  K \subseteq L$ if $a$ is in the image of the induced map 
  $F(K) \arr F(L)$.

  The \emph{essential dimension} $\ed(a)$ of $a \in F(L)$ 
  is the minimum of the transcendence degrees $\trdeg_{k}K$ taken over
  all fields $k \subseteq K \subseteq L$ such that $a$ descends to $K$.

  The essential dimension $\ed(a; p)$ of $a$ at a prime integer $p$
  is the minimum of $\ed(a_{L'})$ taken over all finite field 
  extensions $L'/L$ such that the degree $[L':L]$ is prime to $p$.

The essential dimension $\ed F$ of the functor $F$ (respectively, 
the essential dimension $\ed(F; p)$ of $F$ at a prime $p$) is the
supremum of $\ed(a)$ (respectively, of $\ed(a; p)$)
taken over all $a\in F(L)$ with $L$ in $\Fields_k$.  


Of particular interest to us will be the Galois cohomology functors,
$F_G$ given by $K\leadsto \H^1(K , G)$, where $G$ 
is an algebraic group over $k$.  Here,
as usual, $\H^1(K, G)$ denotes the set of isomorphism classes of
$G$-torsors over $\Spec(K)$, in the fppf topology.  The essential
dimension of this functor is a numerical invariant of $G$, which,
roughly speaking, measures the complexity of $G$-torsors over fields.
We write $\ed G$ for $\ed \, F_G$ and $\ed(G; p)$ for $\ed(F_G; p)$.
Essential dimension was originally introduced in this context;
see~\cite{bur, reichstein, ry}.  The above definition of essential
dimension for a general functor $F$ is due to A.~Merkurjev; see
\cite{bf1}.

Recall that an action of an algebraic group $G$ on an algebraic 
variety $k$-variety $X$ is called ``generically free" 
if $X$ has a dense open subset $U$ such that $\Stab_G(x) = \{ 1 \}$
for every $x \in U(\overline{k})$.

\begin{lemma} \label{lem2.0} Let $G$ be a linear algebraic group
and $V$ be a generically free linear representation of $G$. 
Then $\ed(G) \le \dim(V) - \dim(G)$. 
\end{lemma}

\begin{proof} See~\cite[Theorem 3.4]{reichstein} (which is stated in
characteristic 0 only, but the proof goes through in any characteristic) 
or~\cite[Lemma 4.11]{bf1}.
\end{proof}

\begin{lemma} \label{lem2.1} If $G$ is an algebraic group and
$H$ is a closed subgroup of codimension $e$ then

\smallskip
(a) $\ed(G) \ge \ed(H) - e$, and 

\smallskip
(b) $\ed(G; p) \ge \ed(H; p) - e$ for any prime integer $p$.
\end{lemma}

\begin{proof}
Part (a) is~\cite[Theorem 6.19]{bf1}. Both (a) and (b) follow
directly from~\cite[Principle 2.10]{BrosnanAB}.
\end{proof}

If $G$ is a finite abstract group,  we will write $\ed_k G$ 
(respectively, $\ed_k(G; p)$) for the essential dimension 
(respectively, for the essential dimension at $p$) 
of the constant group $G_k$ over the field $k$.  
Let $\rC(G)$ denote the center of $G$.  

\begin{theorem}\label{t.p-groups} 
Let $G$ be a $p$-group whose commutator $[G, G]$ is central and cyclic.
Then
$ \ed_k(G; p) = \ed_k G = \sqrt{|G/\rC(G)|} + \rank \, \rC(G) - 1$ 
for any base field $k$ of characteristic $\ne p$ containing  
a primitive root of unity of degree equal to the exponent of $G$.
\end{theorem}

Note that with the above hypotheses, $|G/\rC(G)|$ is a complete
square. Theorem~\ref{t.p-groups} was originally proved in~\cite{brv} 
as a consequence of our study of essential dimension of gerbes 
banded by $\mmu_{p^n}$.  Karpenko and Merkurjev~\cite{km2} have 
subsequently refined our arguments to show that the essential 
dimension of any finite $p$-group over any field $k$ containing 
a primitive $p$th root of unity is the minimal dimension of a faithful 
linear $G$-representation defined over $k$. Using~\cite[Remark 4.7]{km2}
Theorem~\ref{t.p-groups} is easily seen to be a special case of their
formula.  For this reason we omit the proof here.

\section{Essential dimension of Spin groups}
\label{sect.spin}

As usual, we will write
$\langle a_1,\ldots, a_n \rangle$ for the quadratic form $q$ of rank $n$
given by $q(x_1,\ldots, x_n)=\sum_{i=1}^n a_i x_i^2$. Let   
\begin{equation} \label{e.h_K}
h = \langle 1, -1 \rangle
\end{equation}
denote the 2-dimensional hyperbolic quadratic form over $k$.
For each $n\ge 0$ we define the $n$-dimensional split form 
$q_n^{\rm split}$ defined over $k$ as follows as follows:
\[ q_n^{\rm split} =   \begin{cases}
  h^{\oplus n/2}, & \text{if $n$ is even,}\\
  h^{\oplus (n-1/2)} \oplus \langle 1 \rangle, & \text{if $n$ is odd.}
\end{cases} \]
Let $\Spin_n \eqdef \Spin(q_n^{\rm split})$ be the split form of 
the spin group.  We will also denote the split forms of the orthogonal 
and special orthogonal groups by $\Orth_n\eqdef\Orth(q_n^{\rm split})$ 
and $\SO_n\eqdef\SO(q_n^{\rm split})$ respectively.
The main result of this section is the following theorem. 

\begin{theorem} \label{thm.spin} (a) Let $k$ be a field 
of characteristic $\ne 2$ and $n \ge 15$ be an integer. 
\[   \ed(\Spin_{n}; 2) \ge \begin{cases}
\text{$2^{(n-1)/{2}} - \frac{n(n-1)}{2}$, if $n$ is odd,} \\    
\text{$2^{(n-2)/{2}} - \frac{n(n-1)}{2}$, if $n \equiv 2 \pmod{4}$, } \\    
\text{$2^{(n-2)/2} - \frac{n(n-1)}{2} + 1$, if $n \equiv 0 \pmod{4}$.}
\end{cases}    
\]
(b) Moreover if $\Char(k) = 0$ then

\smallskip
$\ed(\Spin_{n}) = \ed(\Spin_n; 2) = 2^{(n-1)/{2}} - \frac{n(n-1)}{2}$, 
if $n$ is odd,

\smallskip
   $\ed(\Spin_{n}) = \ed(\Spin_n; 2) = 2^{(n-2)/2} - \frac{n(n-1)}{2}$, 
if $n \equiv 2 \pmod{4}$, and

\smallskip
$\ed(\Spin_{n}; 2) \le \ed(\Spin_{n}) \le  
2^{(n-2)/2} - \frac{n(n-1)}{2} + n$, 
if $n \equiv 0 \pmod{4}$.
\end{theorem}

Note that while the proof of part (a) below goes through for any $n \ge 3$, 
our lower bounds become negative (and thus vacuous) for $n \le 14$.

\begin{proof} (a) Since replacing $k$ by a larger field $k'$ can only decrease
the value of $\ed(\Spin_n; 2)$, we may assume without loss of generality
that $\sqrt{-1} \in k$. The $n$-dimensional split quadratic 
form $q_n^{\rm split}$ is then $k$-isomorphic to
   \begin{equation} \label{e.squares}
   q(x_{1}, \dots, x_{n}) = -(x_{1}^{2} + \dots + x_{n}^{2}).
   \end{equation}
over $k$ and hence, we can write $\Spin_n$ as $\Spin(q)$, $\Orth_n$ as 
$\Orth_n(q)$ and $\SO_n$ as $\SO_n(q)$.

Let $\Gamma_{n} \subseteq \SO_{n}$ be the subgroup consisting 
of diagonal matrices. This subgroup is isomorphic to $\mmu_{2}^{n-1}$. 
Let $G_{n}$ be the inverse image of $\Gamma_{n}$ in 
$\Spin_{n}$; this is a constant group scheme over $k$.  
%
%
%
%
%
By Lemma~\ref{lem2.1}(b) 
\[ \ed(\Spin_n; 2) \ge \ed(G_{n}; 2) - \frac{n(n-1)}{2} \, . \]
Thus in order to prove the lower bounds of part (a),
it suffices to show that 
\begin{equation} \label{e.edG_n}
\ed(G_n; 2) = \ed(G_n) = \begin{cases} \text{$2^{(n-1)/2}$, if $n$ is odd,} \\ 
 \text{$2^{(n-2)/2}$, if $n \equiv 2$ (mod $4$),} \\
 \text{$2^{(n-2)/2} + 1$, if $n$ is divisible by $4$.} \end{cases} 
\end{equation}
The structure of the finite $2$-group $G_{n}$ is well understood; 
see, e.g.,~\cite{wood}. 
Recall that 
the Clifford algebra $A_{n}$ of the quadratic 
form $q$, as in~\eqref{e.squares} is the algebra
given by generators $e_{1}$, \dots,~$e_{n}$, 
and relations $e_{i}^{2} = -1$,
$e_{i}e_{j} + e_{j}e_{i} = 0$ for all $i \neq j$. 
For any $I \subseteq \{1, \dots, n\}$ 
write $I = \{i_{1}, \dots, i_{r}\}$ with 
$i_{1} < i_{2} < \dots < i_{r}$ set 
$e_{I} \eqdef e_{i_{1}} \dots e_{i_{r}}$. 
Here $e_{\emptyset} = 1$.
The group $G_{n}$ consists of the elements of 
$A_{n}$ of the form $\pm e_{I}$, 
where $I \subseteq \{1, \dots, n\}$ and the cardinality
$|I|$ of $I$ is even.  The element $-1$ is central, and 
the commutator $[e_I, e_J]$ is given by
$ [e_{I}, e_{J}] = (-1)^{|I \cap J|}$ .
It is clear from this description that $G_n$ is a $2$-group
of order $2^n$, the commutator subgroup $[G_n, G_n] = \{ \pm 1 \}$ 
is cyclic, and the center $\rC(G)$ is as follows: 
\[ \rC(G_n) = \begin{cases} 
\text{$\{ \pm 1 \} \simeq \ZZ/2 \ZZ$, if $n$ is odd,} \\
\text{$\{ \pm 1, \pm e_{\{ 1, \dots, n \}} \} \simeq \ZZ/4 \ZZ$, 
if $n \equiv 2$ (mod $4$),} \\
\text{$\{ \pm 1, \pm e_{ \{ 1, \dots, n \} } \} \simeq \ZZ/2 \ZZ \times \ZZ/2 \ZZ$, 
if $n$ is divisible by $4$.} \end{cases} \]
Formula~\eqref{e.edG_n} now follows from Theorem~\ref{t.p-groups}. 

\smallskip
(b) Clearly $\ed(\Spin_n; 2) \le \ed(\Spin_n)$. Hence,
we only need to show that for $n \ge 15$
\begin{equation} \label{e.spin-upper-bounds}
\ed(\Spin_{n}) \le \begin{cases}
\text{$2^{(n-1)/{2}} - \frac{n(n-1)}{2}$, if $n$ is odd,} \\    
\text{$2^{(n-2)/{2}} - \frac{n(n-1)}{2}$, if $n \equiv 2 \pmod{4}$,} \\    
\text{$2^{(n-2)/2} - \frac{n(n-1)}{2} + n$, if $n \equiv 0 \pmod{4}$.}
\end{cases}    
\end{equation}
In view of Lemma~\ref{lem2.0} it suffices to show that $\Spin_n$ has
a generically free linear representation $V$ of dimension
\begin{equation} \label{e.reps}
\dim(V) = \begin{cases}
\text{$2^{(n-1)/{2}}$, if $n$ is odd,} \\    
\text{$2^{(n-2)/{2}}$, if $n \equiv 2 \pmod{4}$,} \\    
\text{$2^{(n-2)/2} + n$ if $n \equiv 0 \pmod{4}$.}
\end{cases}    
\end{equation}
Our construction of $V$ will be based on the following 
fact; see~\cite[Theorem 7.11]{pv}.

\begin{lemma} \label{fact.gen-free}
The following representations of $\Spin_n$ over a field $k$ 
of characteristic 0 are generically free:

\smallskip 
(i) the spin representation, if $n \ge 15$ is odd, 

\smallskip 
(ii) either of the two half-spin representation, if 
$n = 4m + 2$ for some integer $m \ge 4$. 
\end{lemma}

\begin{proof} For $n \ge 27$ this follows directly 
from \cite[Theorem 1]{ap}.  For $n$ between $15$ and $25$ 
this is proved in \cite{ampopov}.
\end{proof}

Since the dimension of the spin representation of $\Spin_n$ is
$2^{(n-1)/2}$ (for $n$ odd) and the dimension of the half-spin
representation is $2^{(n-2)/2}$ (for $n$ even),
this completes the proof of part (b) in the cases where $n \ge 15$
is not divisible by $4$.

In the case where $n \ge 16$ is divisible by $4$, 
we define $V$ as the sum of the half-spin representation $W$ of $\Spin_n$
and the natural representation $k^n$ of $\SO_n$, which we will view 
as a $\Spin_n$-representation via 
the projection $\Spin_n \to \SO_n$.  It remains to check that
$V = W \times k^n$ is a generically free representation of $\Spin_n$. 
Indeed, for $a \in k^n$ in general position, $\Stab(a)$ is conjugate 
to $\Spin_{n-1}$ (embedded in $\Spin_n$ in the standard way).
Thus it suffices to show that the restriction of $W$ to $\Spin_{n-1}$
is generically free. Since $W$ restricted to $\Spin_{n-1}$
is the spin representation of $\Spin_{n-1}$ (see, 
e.g.,~\cite[Proposition 4.4]{adams}), and $n \ge 16$,
this follows from Lemma~\ref{fact.gen-free}(i). 
\end{proof}

\begin{remark} \label{rem.char-p} 
The only time the characteristic $0$ assumption is used
in part (b) is in the proof of Lemma~\ref{fact.gen-free}.
It is likely that Lemma~\ref{fact.gen-free} (and thus 
Theorem~\ref{thm.spin}(b)) remain true if $\Char(k) = p > 0$ 
but we have not checked this. 

If $p \ne 2$ and $\sqrt{-1} \in k$, we have the weaker (but 
asymptotocally equivalent)
upper bound $\ed(\Spin_n) \le \ed(G_n)$, where $\ed(G_n)$
is given by~\eqref{e.edG_n}. This is a consequence of the fact that
every $\Spin_n$-torsor admits reduction of structure to
$G_n$, i.e., the natural map $\H^1(K, G_n) \to \H^1(K, \Spin_n)$
is surjective for every field $K/k$; cf.~\cite[Lemma 1.9]{bf1}.
\end{remark}

\begin{remark} \label{rem.cs}
The best previously known lower bounds on $\ed(\Spin_n)$,
due of V.~Chernousov and J.--P.~Serre~\cite{cs}, were
\[ \ed(\Spin_n; 2) \geq \begin{cases} \lfloor n/2 \rfloor + 1 &\text{if $n
    \ge 7$
    and $n \equiv 1$, $0$ or $-1 \pmod{8}$} \\
  \lfloor n/2 \rfloor &\text{for all other $n \ge 11$.} \end{cases} \] 
(The first line is due to B.~Youssin and the second author 
in the case that $\chr k=0$~\cite{ry}.)  Moreover, in low 
dimensions, M.~Rost~\cite{rost} 
computed the following table of exact values
\begin{center} \renewcommand{\arraystretch}{1.25}
\rule{0pt}{14pt}
\begin{tabular}{r@{${}={}$}lr@{${}={}$}lr@{${}={}$}lr@{${}={}$}l}
$\ed \Spin_3$ & $0$ &$ \ed \Spin_4$ & $0$ & $\ed \Spin_5$ & $0$ &
   $\ed \Spin_6$ & $0$ \\
$\ed \Spin_7$ & $4$ & $\ed \Spin_8$ & $5$ & $\ed \Spin_9$ & $5$ &
   $\ed \Spin_{10}$ & $4$\\
$\ed \Spin_{11}$ & $5$ & $\ed \Spin_{12}$ & $6$ & $\ed
\Spin_{13}$ & $6$ &
   $\ed \Spin_{14}$ & $7$;
\end{tabular}
\end{center}
cf. also~\cite{garibaldi}. 
Taken together these results seemed to suggest
that $\ed\Spin_n$ should be a slowly increasing
function of $n$ and gave no hint of its exponential growth.
\end{remark}

\begin{remark} A. S. Merkurjev (unpublished) recently strengthened our lower 
bound on $\ed(\Spin_n;2)$, in the case where $n \equiv 0 \pmod{4}$ as follows:
\[ \ed(\Spin_n;2) \ge 2^{(n-2)/2} - \frac{n(n-1)}{2} + 2^m \, , \]
where $2^m$ is the highest power of $2$ dividing $n$. If $n \ge 16$ is a power
of $2$ and $\Char(k) = 0$ this, in combination with the upper bound of 
Theorem~\ref{thm.spin}(b), yields 
\[ \ed(\Spin_n; 2) = \ed(\Spin_n) =  
2^{(n-2)/2} - \frac{n(n-1)}{2} + n \, . \]
In particular, $\ed(\Spin_{16}) = 24$. The first value of $n$ 
for which $\ed(\Spin_n)$ is not known is $n = 20$, where 
$326 \le \ed(\Spin_{20}) \le 342$. 
\end{remark}

\begin{remark} \label{rem.pin} 
The same argument can be applied to the half-spin groups
yielding
   \[ \ed(\HSpin_{n}; 2) = \ed(\HSpin_{n}) = 2^{(n-2)/2} - \frac{n(n-1)}{2} \]
for any integer $n \ge 20$ divisible by $4$ over any field 
of characteristic $0$.
Here, as in Theorem~\ref{thm.spin}, the lower bound 
\[ \ed(\HSpin_{n}; 2) \ge 2^{(n-2)/2} - \frac{n(n-1)}{2} \]
is valid for over any base field $k$ of characteristic $\ne 2$. 
The assumptions that $\Char(k) = 0$ and
$n \ge 20$ ensure that the half-spin representation 
of $\HSpin_{n}$ is generically free; see~\cite[Theorem 7.11]{pv}. 
\end{remark}

\begin{remark}
Theorem~\ref{thm.spin} implies that for large $n$,
$\Spin_n$ is an example of a split, semisimple,
connected linear algebraic group whose essential 
dimension exceeds its dimension. Previously no examples 
of this kind were known, even for $k = \mathbb{C}$.

Note that no complex connected semisimple adjoint
group $G$ can have this property. Indeed, every such $G$
has a generically free representation
$V = \mathfrak{g} \times \mathfrak{g}$, where 
$\mathfrak{g}$ is the adjoint representation of $G$ on its
Lie algebra; cd., e.g.,~\cite[Lemma 3.3(b)]{richardson}.
Thus $\ed G \le \dim (G)$ by Lemma~\ref{lem2.0}.
\end{remark}

\begin{remark} Since $\ed\SO_n = n - 1$ for every $n \ge 3$
(cf.~\cite[Theorem 10.4]{reichstein}), for large $n$, $\Spin_n$ 
is also an example of a split, semisimple, connected linear
algebraic group $G$ with a central subgroup $Z$ such
that $\ed G>\ed G/Z$. To the best of our knowledge, 
this example is new as well.
\end{remark}

\section{Pfister numbers}
\label{s.arason}

Let $K$ be a field of characteristic not equal to $2$ 
and $a \ge 1$ be an integer. We will continue to denote the Witt ring
of $K$ by $W(K)$ and its fundamental ideal by $I(K)$. If non-singular
quadratic forms $q$ and $q'$ over $K$ are Witt equivalent, we will write
$q \sim q'$.

As we mentioned in the introduction, the $a$-fold Pfister forms
generate $I^a(K)$ as an abelian group. In other words,
every $q \in I^a(K)$ is Witt equivalent to
$
\sum_{i=1}^r \pm p_i, 
$
where each $p_i$ is an $a$-fold Pfister form over $K$.
We now define the \emph{$a$-Pfister number} of $q$ to be
the smallest possible number $r$ of Pfister forms 
appearing in any such sum.  The \emph{$(a,n)$-Pfister number} 
$\Pf_k(a, n)$ is the supremum of the
$a$-Pfister number of $q$, taken over all field extensions 
$K/k$ and all $n$-dimensional forms $q \in I^a(K)$.   

\begin{proposition}
\label{p.PfisterEasy}  
Let $k$ be a field of characteristic 
  $\neq 2$ and let $n$ be a positive even integer. 
  \begin{enumeratea}
  \item $\Pf_k(1,n) \leq n$.
  \item $\Pf_k(2,n) \leq n - 2$.
  \end{enumeratea}
\end{proposition}
  
\begin{proof}
  (a) Immediate from the identity 
\[ \langle a_1, a_ 2 \rangle \sim
\langle 1, a_1 \rangle - \langle 1, - a_2 \rangle = \da{-a_1} - \da{a_2} \]
in the Witt ring.

  (b) Let $q=\langle a_1,\ldots, a_n\rangle$ be an $n$-dimensional
quadratic form over $K$. Recall that 
  $q \in I^2(K)$ iff $n$ is even and $d_{\pm}(q)=1$, modulo 
$(K^*)^2$~\cite[Corollary II.2.2]{lam}. Here $d_{\pm}(q)$ 
is the signed discriminant
  given by $(-1)^{n(n-1)/2}d(q)$ where $d(q)=\prod_{i=1}^n a_n$ is the
  discriminant of $q$; cf.~\cite[p. 38]{lam}.  

To explain how to write $q$ as a sum of $n-2$ Pfister forms,
we will temporarily assume that $\sqrt{-1} \in K$. In this case 
we may assume that $a_1 \dots a_n = 1$. Since $\langle a, a \rangle$
is hyperbolic for every $a \in K^*$, we see that
$q = \langle a_1, \dots, a_n \rangle$ is Witt equivalent to
\[ \ll a_2, a_1 \gg \oplus 
\ll a_3, a_1a_2 \gg \oplus \cdots \oplus \ll a_{n-1}, a_1 \dots a_{n-2} \gg \, . \]
By inserting appropriate powers of $-1$, we can modify this formula so that
it remains valid even if we do not assume that $\sqrt{-1} \in K$, as follows:
\[ q = \langle a_1, \dots, a_n \rangle \sim
\sum_{i=2}^n (-1)^i \da{(-1)^{i+1} a_i, (-1)^{i(i-1)/2 +1} a_1 \dots a_{i-1}} 
\qedhere \]
\end{proof}

We do not have an explicit upper bound on $\Pf_k(3, n)$; 
however, we do know that $\Pf_k(3, n)$ is finite for any $k$ and any $n$.
To explain this, let us recall that $I^3(K)$ is the set of all classes
$q \in\rW(K)$ such that $q$ has even dimension, trivial signed 
discriminant and trivial Hasse-Witt invariant~\cite{invol}.  
The following result was suggested to us by Merkurjev and Totaro.

\begin{proposition}
\label{p.PfisterEasy2} Let $k$ be a field of characteristic 
different from~$2$.  Then $\Pf_k(3,n)$ is finite.
\end{proposition}

\begin{proof}[Sketch of proof]
  Let $E$ be a versal torsor for $\Spin_n$ over a field extension
  $L/k$; cf. \cite[Section I.V]{gms}.
  Let $q_L$ be the quadratic form over $L$ corresponding to
  $E$ under the map $\H^1(L,\Spin_{n})\to \H^1(L,\Orth_{n})$.  The
  $3$-Pfister number of $q_L$ is then an upper bound for the
  $3$-Pfister number of any $n$-dimensional form in $I^3$ 
over any field extension $K/k$.
\end{proof}

\begin{remark}
  For $a > 3$ the finiteness of $\Pf_k(a,n)$ is an open problem.
\end{remark}

\section{Proof of Theorem~\ref{t.pfister}}
The goal of this section is to prove Theorem~\ref{t.pfister}
stated in the introduction, which says, in particular, that
   \[ \Pf_k(3, n) \ge \frac{2^{(n+4)/4} -n-2}{7}  \]
for any field $k$ of characteristic different from $2$
and any positive even integer $n$.
Clearly, replacing $k$ by a larger field $k'$ strengthens
the assertion of Theorem~\ref{t.pfister}. Thus, 
we may assume without loss of generality that $\sqrt{-1} \in k$.
This assumption will be in force for the remainder of this section.

For each extension $K$ of $k$, denote by $\rT_n(K)$ the image of
$\H^{1}(K, \Spin_{n})$ in $\H^{1}(K, \SO_{n})$. 
We will view $\rT_{n}$ as a functor $\Fields_{k} \arr \Sets$. 
Note that $\rT_n(K)$ is the set of isomorphism classes 
of $n$-dimensional quadratic forms $q \in I^3(K)$.

\begin{lemma} \label{lem.T-Spin}
We have the following inequalities:
\begin{enumeratea} 
\item $\ed\Spin_{n} - 1 \leq \ed\rT_{n} \leq \ed\Spin_{n}$,
\item $\ed(\Spin_{n}; 2) - 1 \leq \ed(\rT_{n}; 2) \leq \ed(\Spin_{n}; 2)$.
\end{enumeratea}
\end{lemma}

\begin{proof}
  In the language of~\cite[Definition 1.12]{bf1}, we 
have a fibration of functors
\[
\H^1(\ast , \mmu_{2})\leadsto \H^1(\ast ,\Spin_n)\arr \rT_{n}(\ast).
\]
The first inequality in part (a) follows from~\cite[Proposition 1.13]{bf1}
and the second from Proposition~\cite[Lemma 1.9]{bf1}.
The same argument proves part (b).
\end{proof}

Let $K/k$ be a field extension. Let
$h_{K} = \langle 1, -1 \rangle$ be the 
$2$-dimensional hyperbolic form over $K$. (Note in Section~\ref{sect.spin} 
we wrote $h$ in place of $h_k$; see~\eqref{e.h_K}.)  
For each $n$-dimensional quadratic form $q \in I^3(K)$, 
let $\ed_{n}(q)$ denote the essential dimension 
of the class of $q$ in $\rT_{n}(K)$.

\begin{lemma}\label{lem:decrease-ed}
Let $q$ be an $n$-dimensional quadratic form in $I^3(K)$. Then 
   \[ \ed_{n+2s}(h_{K}^{\oplus s} \oplus q)\geq \ed_{n}(q)
      - \frac{s(s + 2n -1)}{2} \]
for any integer $s \ge 0$. 
\end{lemma}

\begin{proof}
  Set $m \eqdef \ed_{n+2s}(h_{K}^{\oplus s} \oplus q)$. By definition, 
  $h^{\oplus s}_{K} \oplus q$ descends to an intermediate subfield
  $k \subset F \subset K$ such that $\trdeg_k(F) = m$.
  In other words, there is an $(n + 2s)$-dimensional quadratic 
  form $\widetilde{q} \in I^3(F)$ 
  such that $\widetilde{q}_K$ is $K$-isomorphic
  to $h_{K}^{\oplus s} \oplus q$. Let $X$ be the Grassmannian of
  $s$-dimensional subspaces of $F^{n+2s}$ which are totally isotropic
  with respect to $\widetilde{q}$.  The dimension of $X$ over $F$ is 
  $s(s + 2n -1)/2$.

  The variety $X$ has a rational point over $K$; hence there exists an
  intermediate extension $F \subseteq E \subseteq K$ such that
  $\trdeg_{F}E \leq s(s + 2n -1)/2$, with the property that
  $\widetilde{q}_{E}$ has a totally isotropic subspace of
  dimension~$s$. Then $\widetilde{q}_{E}$ splits as $h_{E}^{s} \oplus
  q'$, where $q' \in I^3(E)$.
  By Witt's Cancellation Theorem, $q'_K$ is $K$-isomorphic to $q$; 
  hence 
\[ \ed_{n}(q) \le \trdeg_k \, E = \trdeg_k \, F + \trdeg_F \, E =  
m + s(s + 2n -1)/2 \, , \]
as claimed.
\end{proof}

We now proceed with the proof of Theorem~\ref{t.pfister}.
For $n \leq 10$ the statement of the theorem 
is vacuous, because $2^{(n+4)/4}-n-2 \leq 0$. 
Thus we will assume from now on that $n \geq 12$. 

Lemma~\ref{lem.T-Spin} implies, in particular, that
$\ed(\rT_{n}; 2)$ is finite. Hence, there exist a field
$K/k$ and an $n$-dimensional form $q \in I^3(K)$ such 
that $\ed_{n}(q) = \ed(\rT_{n}; 2)$. We will show that 
this form has the properties asserted by Theorem~\ref{t.pfister}.
In fact, it suffices to prove that if $q$ is Witt equivalent to
\[
\sum_{i=1}^r \da{a_i,b_i,c_i}.
\]  
over $K$ then $r \ge \dfrac{2^{(n+4)/4} -n-2}{7}$. 
Indeed, $\ed_{n}(q_L) = \ed(\rT_{n}; 2)$ for any finite odd degree 
extension $L/K$. Thus if we can prove the above inequality for $q$, 
it will also be valid for $q_L$. 

Let us write a 3-fold Pfister form $\da{a,b,c}$ as
   $\da{a,b,c} = \generate{1} \oplus \da{a,b,c}_{0}$,
where
   \[
   \da{a,b,c}_{0} \eqdef
   \langle a_i, b_i, c_i, a_ib_i, a_i c_i, b_i c_i, a_i b_i c_i \rangle.
   \]
Set
   \[
   \phi \eqdef \begin{cases} 
 \text{$\sum_{1 = 1}^{r}\da{a_{i}, b_{i}, c_{i}}_{0}$, if $r$ is even, and} \\
   \text{$\generate{1} \oplus \sum_{1 = 1}^{r}\da{a_{i}, b_{i}, c_{i}}_{0}$,
if $r$ is odd.} \end{cases}
   \]
Then $q$ is Witt equivalent to 
$\phi$ over $K$; in particular, $\phi \in I^3(K)$. The dimension 
of $\phi$ is $7r$ or $7r+1$, depending on the parity of $r$. 

We claim that $n < 7r$. Indeed, assume the contrary. Then 
$\dim(q) \le \dim(\phi)$, so that $q$ is isomorphic to 
a form of type $h_{K}^{s} \oplus \phi$ over $K$. Thus 
\[  \frac{3n}{7} \geq 3r \geq\ed_{n}(q) = \ed(\rT_{n}; 2) 
\stackrel{\text{\tiny by Lemma~\ref{lem.T-Spin}}}{\ge} 
 \ed(\Spin_{n}; 2) - 1 \, . \]
The resulting inequality fails for every even $n \ge 12$ because for 
such $n$ \[ \ed(\Spin_n; 2) \ge n/2; \] 
see~Remark~\ref{rem.cs}.

So, we may assume that $7r > n$, i.e., 
$\phi$ is isomorphic to $h_{K}^{\oplus s}\oplus q$ over $K$,
for some $s \ge 1$.  By comparing dimensions we get 
the equality $7r = n + 2s$ when $r$ is even, 
and $7r+1 = n + 2s$ when $r$ is odd.
The essential dimension of the form $\phi$, as an element 
of $\rT_{7r}(K)$ or $\rT_{7r+1}(K)$ is at most $3r$, 
while Lemma~\ref{lem:decrease-ed} tells us that this essential 
dimension is at least $\ed_n(q) - s(s + 2n -1)/2$. 
From this, Lemma~\ref{lem.T-Spin} 
and Theorem~\ref{thm.spin}(a) we obtain the following chain of inequalities
   \begin{align} 
   3r &\geq \ed_{n}(q) - \frac{s(s + 2n -1)}{2} 
   = \ed(\rT_{n}; 2) - \frac{s(s + 2n -1)}{2} \nonumber \\
   & \geq \ed\Spin_{n} - 1 - \frac{s(s + 2n -1)}{2}  \label{e.quadratic} \\
   & \geq 2^{(n-2)/2} - \frac{n(n-1)}{2} - 1 - \frac{s(s + 2n -1)}{2}. \nonumber
   \end{align}

Now suppose $r$ is even. Substituting $s = (7r - n)/2$ into 
inequality~\eqref{e.quadratic}, we obtain
   \[
   \frac{49r^{2} + (14n+10)r - 2^{(n+4)/2} - n^{2} + 2n - 8}{8} \geq 0.
   \]
We interpret the left hand side as a quadratic polynomial in $r$. 
The constant term of this polynomial is negative
for all $n \geq 8$; hence this polynomial has one positive real 
root and one negative real root. Denote the positive root by $r_+$.
The above inequality is then equivalent to $r \geq r_+$. 
By the quadratic formula 
   \[ r_+ = \frac{\sqrt{49 \cdot 2^{(n+4)/2} + 168n - 367} -(7n + 5)}{49}
\ge
      \frac{2^{(n+4)/4} -n-2}{7}\, . \]
   This completes the proof of Theorem~\ref{t.pfister} when $r$ is
   even. If $r$ is odd then substituting $s = (7r+1-n)/2$ 
   into~\eqref{e.quadratic},
   we obtain an analogous quadratic inequality whose positive root is
   \[ r_+ = \frac{\sqrt{49 \cdot 2^{(n+4)/2} + 168n - 199} -(7n + 12)}{49}\\
      \geq  \frac{2^{(n+4)/4} -n-2}{7}\, , \] 
and Theorem~\ref{t.pfister} follows.
\qed

\bibliographystyle{amsalpha}
\bibliography{ed}
\end{document}
\def\cprime{$'$}
\providecommand{\bysame}{\leavevmode\hbox to3em{\hrulefill}\thinspace}
\providecommand{\MR}{\relax\ifhmode\unskip\space\fi MR }
\providecommand{\MRhref}[2]{%
  \href{http://www.ams.org/mathscinet-getitem?mr=#1}{#2}
}
\providecommand{\href}[2]{#2}

\end{document}